\documentclass[12pt, reqno]{amsart}
\usepackage{amsmath}
\usepackage{amsthm}
\usepackage{amssymb}
\usepackage{amsrefs}
\usepackage{mathrsfs}
\usepackage[usenames]{color}
\usepackage[utf8]{inputenc}
\usepackage{latexsym}
\usepackage{yhmath}
\usepackage{graphicx}
\usepackage{ifthen}

\usepackage{hyperref}

\newtheorem{thm}{}[section]
\newtheorem{theorem}[thm]{Theorem}

\newtheorem{definition}[thm]{Definition}
\newtheorem{lemma}[thm]{Lemma}
\newtheorem{proposition}[thm]{Proposition}

\theoremstyle{remark}
\newtheorem{remark}[thm]{Remark}

\newtheorem{question}[thm]{Question}

\numberwithin{equation}{section}
\allowdisplaybreaks

\newcommand{\Env}[2][]{%
\ifthenelse{ \equal{#1}{} }  
{\ensuremath{#2_{\mathsf{c}}}}  
{\ensuremath{#2_{\mathsf{c},#1}}}
}

\newcommand{\ee}{\ensuremath{\mathbf{e}}}

\newcommand{\BB}{\ensuremath{\mathcal{B}}}

\newcommand{\FF}{\ensuremath{\mathbb{F}}}

\newcommand{\RR}{\ensuremath{\mathbb{R}}}
\newcommand{\NN}{\ensuremath{\mathbb{N}}}

\newcommand{\XX}{\ensuremath{\mathbb{X}}}

\DeclareMathOperator{\supp}{supp}
\DeclareMathOperator{\sgn}{sign}

\author[P.\ M. Bern\'a]{Pablo M. Bern\'a}
\address{Pablo M. Bern\'a\\
	Departamento de Matem\'atica Aplicada y Estad\'istica, Facultad de Ciencias Econ\'omicas y Empresariales, Universidad San Pablo-CEU, CEU Universities\\ Madrid, 28003 Spain.}
\email{pablo.bernalarrosa@ceu.es}

\subjclass[2010]{46B15, 41A65}

\keywords{quasi-Banach spaces, greedy bases, quasi-greedy bases, greedy algorithm}

\usepackage{vmargin}

\setmargins{2.5cm}       
{1.5cm}                        
{16.5cm}                      
{23.42cm}                    
{10pt}                           
{1cm}                           
{0pt}                             
{2cm}                           
\begin{document}

\title[A note on partially-greedy bases in quasi-Banach spaces]{A note on partially-greedy bases in quasi-Banach spaces}

\begin{abstract} 
	We continue with the study of greedy-type bases in quasi-Banach spaces started in \cite{AABW}. In this paper, we study partially-greedy bases focusing our attention in two main results: 
	\begin{itemize}
		\item Characterization of partially-greedy bases in quasi-Banach spaces in terms of quasi-greediness and different conservative-like properties.
		\item Given a $C$-partially-greedy basis in a quasi-Banach space, there exists a ``renorming" such that the basis is $1$-partially-greedy. 
	\end{itemize}
\end{abstract}

\thanks{P.M. Bern\'a was partially supported by the grants MTM-2016-76566-P (MINECO, Spain) and 20906/PI/18 from Fundaci\'on S\'eneca (Regi\'on de Murcia, Spain).}

\maketitle
\section{Introduction}

For several years, how to represent a function $f$ in a  particular space $\XX$, has been considered a quite important problem in the area of Mathematical Analysis, that is, given $f\in\XX$, we want to obtain

$$f=\sum_{n=1}^\infty a_n \ee_n,$$
for a given basic functions $(\ee_n)_{n=1}^\infty$ and for a suitable scalars $a_n$ (coefficients).  In literature we can find different examples of such representations as for instance Fourier series of functions or Taylor expansions. In Functional Analysis, one consider expansions with regard to a basis, that is, we could consider that $\mathcal B=(\ee_n)_{n=1}^\infty$ is a Schauder or a Markushevich basis. 

One of the main goals of Approximation Theory is to find good approximations of $f$ in terms of finite sums. Concretely, we want to find algorithms of approximation  $(T_m)_{m=1}^\infty$, where 
$$T_m(f)=\sum_{n\in A}b_n \ee_n,$$
$A$ is a finite set of cardinality $m$ and $b_n$ could be different from the original coefficients of $f$. Moreover, we would like to have that $(T_m)_{m=1}^\infty$ produces a ``good approximation" (that could be interpreted taking into account our preferences).

Since 1999, in the field of Non-Linear Approximation, one of the most important algorithms that several researchers have studied is the  Greedy Algorithm $(G_m)_{m=1}^\infty$, where for an element $f$ in $\XX$, the algorithm selects the biggest coefficients of $f$ in modulus. This algorithm was introduced by S. V. Konyagin and V. N. Temlyakov in \cite{KoTe1999} and different and new properties about greedy approximation have been analyzed by different authors, among them S. J. Dilworth, N. J. Kalton, D. Kutzarova, P. Wojtaszczyk, etc (see \cite{DKK2003,DKKT2003,Wo2000}). In this paper, we focus our attention in the following property that was introduced in \cite{DKKT2003}: there exists a positive constant $C$ such that
\begin{eqnarray}\label{in}
\Vert f-G_m(f)\Vert \leq C\Vert f-S_m(f)\Vert,\; \forall f\in\XX, \forall m\in\NN,
\end{eqnarray}
where $S_m$ denotes the $m$th partial sum. The importance of this property resides in the fact that we want to compare if the non-linear approximation is better than  linear approximation. Here, we extend the main known results about \eqref{in} in Banach spaces for the case of quasi-Banach spaces, using recent results proved in \cite{AABW}.

The structure of the paper is the following: in Section \ref{two} we give the basic definitions about bases in quasi-Banach spaces and some operators. In Section \ref{thre} we give the definition of the Greedy Algorithm, we talk about some greedy-type bases and we analyze some properties about conservativeness. In Section \ref{four} we give the main characterization of partially-greedy bases in quasi-Banach spaces and, finally, in Section \ref{five}, we talk about renorming of quasi-Banach spaces using greedy-type bases.

\section{Preliminaries on quasi-Banach spaces}\label{two}

We say that the map $\Vert \cdot\Vert : \mathbb X\rightarrow\mathbb [0,+\infty)$ defined on a vector space $\XX$ over $\mathbb F=\RR$ or $\mathbb C$  is a \textbf{quasi-norm} if

\begin{itemize}
	\item[a)] $\Vert f\Vert>0$ for all $f\neq 0$,
	\item[b)] $\Vert t\, f\Vert = \vert t\vert \Vert f\Vert$, for all $t\in\FF$ and $f\in\XX$,
	\item[c)] there exists a positive constant $k$ such that for all $f,g\in\XX$,
	$$\Vert f+g\Vert \leq k\left(\Vert f\Vert+\Vert g\Vert\right).$$
\end{itemize}

Given $p\in (0,1]$, a \textbf{$p$-norm} is a map $\Vert \cdot\Vert : \XX\rightarrow [0,+\infty)$ satisfying a), b) and
\begin{itemize}
	\item[d)] $\Vert f+g\Vert^p \leq \Vert f\Vert^p + \Vert g\Vert^p$, for all $f,g\in\XX$.
\end{itemize}

Of course, d) implies c) with constant $k=2^{1/p-1}$. If $\Vert \cdot\Vert$ is a quasi-norm (resp. $p$-norm) on $\XX$ such that defines a complete metrizable topology, then $\XX$ is called \textbf{quasi-Banach space} (resp. \textbf{$p$-Banach space}). Thanks to the Aoki-Rolewicz's Theorem (see \cite{Aoki},\cite{Rolewicz}), any quasi-Banach space $\XX$ is $p$-convex, that is, 
$$\left\Vert \sum_{i=1}^n x_j\right\Vert \leq C\left(\sum_{j=1}^n \Vert x_j\Vert^p\right)^{1/p},\; n\in\mathbb N, x_j\in\mathbb X.$$
This way, $\mathbb X$ becomes $p$-Banach under a suitable renorming, for some $0<p\leq 1$. 
\subsection{Bases} 
Throughout this paper, a \textbf{basis} in a quasi-Banach space $\XX$ is a \textit{semi-normalized and total Markushevich basis}, i.e., $\mathcal B=(\ee_n)_{n=1}^\infty\subset\XX$ verifies the following conditions:
\begin{itemize}
	\item[i)] $[\ee_n : n\in\NN]=\XX$, (completion)
	\item[ii)] there exists a unique sequence $(\ee_n^*)_{n=1}^\infty\subset\XX^*$, called \textbf{biorthogonal functionals}, such that $\ee_n^*(\ee_m)=\delta_{n,m}$,
	\item[iii)] if $\ee_n^*(f)=0$ for all $n\in\NN$, then $f=0$, (totality) 
	\item[iv)] $\sup_n\max\lbrace\Vert \ee_n\Vert,\Vert\ee_n^*\Vert\rbrace<\infty$ (semi-normalization).
\end{itemize}

Thanks to \cite[Lemma 1.8]{AABW}, we now that if $f\in\XX$, $f=\sum_{n=1}^{\infty}\ee_n^*(f)\ee_n$, where  $\lim_{n\rightarrow+\infty}\ee_n^*(f)=0$ and the expansion is only formal but the assignment of coefficients is still unique. Also, we denote by $\supp(f)=\lbrace n\in\mathbb N : \ee_n^*(f)\neq 0\rbrace$.

Associated with a basis, we can consider the \textbf{projection operator} $P_A$, where for a finite set $A\subset\mathbb N$, 
$$P_A(f)=\sum_{n\in A}\ee_n^*(f)\ee_n.$$
It is well known that if $P_A$ is uniformly bounded, then $\mathcal B$ is called unconditional. 
We write $S_k := P_{\lbrace 1,...,k\rbrace}$ as the \textbf{$m$th partial sum}. Also, if there is a positive constant $C$ such that
$$\Vert S_k(f)\Vert\leq C \Vert f\Vert,\; \forall f\in\XX, \forall k\in\NN,$$
we say that $\mathcal B$ is a \textbf{Schauder basis}. Finally, if $A$ is a finite set, we denote by $\Psi_A$ the collection of signs $\varepsilon$:

$$\Psi_A:=\lbrace \varepsilon=(\varepsilon_n)_{n\in A} : \vert \varepsilon_n\vert=1,\; n\in A\rbrace,$$
and 
$\mathbf{1}_{\varepsilon A}[\mathcal B,\mathbb X]:=\mathbf{1}_{\varepsilon A}$ is the indicator sum:
$$\mathbf{1}_{\varepsilon A}=\sum_{n\in A}\varepsilon_n \ee_n,$$
where $\varepsilon \in\Psi_A$. If $\varepsilon\equiv 1$, we write $\mathbf{1}_A$. As usual, if $A, B\subset\mathbb N$ are finite sets, we denote by $A<B$ if $\max A <\min B$.
\subsection{$p$-convexity}
Consider, for $0<p\leq 1$, the following geometrical constants as in \cite{AABW}:
$$\mathbf{A}_p=\frac{1}{(2^p-1)^{1/p}},$$
and
$$\mathbf{B}_p = \left\{
\begin{array}{c l}
2^{1/p}\mathbf{A}_p & \text{if}\; \mathbb F=\RR,\\
4^{1/p}\mathbf{A}_p & \text{if}\; \mathbb F=\mathbb C.
\end{array}
\right.
$$
The following result is a a collection of two corollaries of \cite[Theorem 1.2]{AABW}, where this theorem  plays the role of a substitute of the Bochner integral in the case of $p$-Banach spaces, for $0<p\leq 1$.

\begin{proposition}[{\cite[Corollaries 1.3 and 1.4]{AABW}}]\label{conv}
Let $\mathcal B=(\ee_n)_{n=1}^\infty$ a basis in a $p$-Banach space and $J$ a finite set. Then,
\begin{itemize}
	\item[a)] For any scalars $(a_n)_{n\in J}$ with $0\leq a_n\leq 1$ and any $g\in\XX$, we have
\begin{eqnarray*}
\left\Vert g+\sum_{n\in J}a_n \ee_n\right\Vert\leq \mathbf{A}_p\sup\left\lbrace\left\Vert g+\sum_{n\in A}\ee_n\right\Vert : A\subseteq J\right\rbrace.
\end{eqnarray*}

\item[b)] For any scalars $(a_n)_{n\in J}$ with $\vert a_n\vert\leq  1$ and any $g\in\XX$, we have
\begin{eqnarray*}
\left\Vert g+\sum_{n\in J}a_n \ee_n\right\Vert\leq \mathbf{A}_p\sup\left\lbrace\left\Vert g+\sum_{n\in J}\varepsilon_n\ee_n\right\Vert : \vert\varepsilon_n\vert=1\right\rbrace.
\end{eqnarray*}

\item[c)] For any scalars $(a_n)_{n\in J}$ with $\vert a_n\vert\leq 1$, we have
$$\left\Vert \sum_{n\in J} a_n \ee_n\right\Vert \leq \mathbf{B}_p\sup_{A\subseteq J}\left\Vert\sum_{n\in A}\ee_n\right\Vert.$$
\end{itemize}
\end{proposition}
\section{The Greedy Algorithm and greedy-type bases}\label{thre}

For each $f\in \XX$ and $m\in \NN$, S. V. Konyagin and V. N. Temlyakov defined in \cite{KoTe1999} a \textbf{greedy sum} of $f$ of order $m$ by $$G_m(f) = \sum_{n=1}^me_{\pi(n)}^{*}(f)e_{\pi(n)},$$ where $\pi$ is a \textbf{greedy ordering}, that is, $\pi: \mathbb{N}\longrightarrow\mathbb{N}$  is a permutation such that $\supp(f)\subseteq \pi(\mathbb{N})$ and $\vert e^*_{\pi(i)}(f)\vert \geq \vert e^*_{\pi(j)}(f)\vert$ for $i\leq j$. The series $\sum_{n=1}^\infty \ee_{\pi(n)}^*(f)\ee_{\pi(n)}$ is called the \textbf{greedy series}. Also, $G_m(f)=P_A(f)$, where the set $A=\supp(G_m(f))$ is called a \textbf{greedy set} and satisfies that $\vert A\vert=m$ and 
$$\min_{n\in A}\vert\ee_n^*(f)\vert\geq \max_{n\not \in A}\vert\ee_n^*(f)\vert.$$
%

\subsection{Quasi-greedy and partially-greedy bases}

To study the convergence of the greedy algorithm, we consider the following bases introduced by S. V. Konyagin and V. N. Temlyakov.

\begin{definition}[\cite{KoTe1999}]
We say that a basis $\mathcal B$ in a quasi-Banach space $\XX$ is \textbf{quasi-greedy} if there is a positive constant $C$ such that for all $f\in\XX$,
\begin{eqnarray}\label{qg}
\Vert P_A(f)\Vert\leq C\Vert f\Vert,
\end{eqnarray}
whenever $A$ is a finite greedy set of $f$. The smallest constant verifying \eqref{qg} is called the \textbf{quasi-greedy constant} of the basis, it is denoted by $C_{qg}=C_{qg}[\mathcal B, \XX]$ and we say that $\mathcal B$ is $C_{qg}$-quasi-greedy.
\end{definition}

The following result was proved in \cite{Wo2000} and \cite{AABW}:

\begin{theorem}
	Let $\mathcal B$ a basis in a quasi-Banach space $\XX$. The following are equivalent:
	\begin{itemize}
		\item $\mathcal B$ is quasi-greedy.
		\item For every $f\in\XX$, the greedy series of $f\in\XX$ converges.
	\end{itemize}
\end{theorem}

Consider the following weaker version of quasi-greediness that we need for our purposes.
\begin{definition}[\cite{AABW}]
	We say that a basis $\mathcal B$ in a quasi-Banach space $\XX$ is \textbf{quasi-greedy for largest coefficients} if there exists a positive constant $C$ such that
	\begin{eqnarray}\label{ql}
	\Vert \mathbf{1}_{\varepsilon A}\Vert\leq C\Vert f+\mathbf{1}_{\varepsilon A}\Vert,
	\end{eqnarray}
	for any finite set $A\subset \NN$, any $\varepsilon\in\Psi_A$ and any $f\in\XX$ such that $\supp(f)\cap A=\emptyset$ and $\max_{n\in \supp(f)}\vert\ee_n^*(f)\vert\leq 1$. The smallest constant verifying \eqref{ql} is denoted by $C_{ql}=C_{ql}[\BB,\XX]$ and we say that $\mathcal B$ is $C_{ql}$-quasi-greedy for largest coefficients. 
	Of course, $C_{ql}\leq C_{qg}$. 
\end{definition}

Since 1999, different types of convergence of the greedy algorithm have been studied by different authors introducing different greedy-type bases such as greedy bases and almost-greedy bases. These bases were introduced and characterized (in the context of Banach spaces) in \cite{KoTe1999, DKKT2003}. Recently, in \cite{AABW}, the authors characterized the same bases in the context of quasi-Banach spaces analyzing the lack of convexity in the results. Now, we study the characterization of partially-greedy bases.

\begin{definition}
	We say that a basis $\mathcal B$ in a quasi-Banach space $\XX$ is \textbf{partially-greedy} if there is a positive constant $C$ such that, for all $f\in\mathbb X$ and all finite greedy set $A$ of $f$,
	\begin{eqnarray}\label{pg}
	\Vert f-P_A(f)\Vert \leq C\inf_{k\leq \vert A\vert}\Vert f-S_k(f)\Vert.
	\end{eqnarray}
	The smallest constant verifying \eqref{pg} is called the \textbf{partially-greedy constant} of the basis, it is denoted by $C_{pg}=C_{pg}[\mathcal B, \XX]$ and we say that $\BB$ is $C_{pg}$-partially-greedy.
\end{definition}

\begin{remark}
In \cite{DKKT2003}, the authors defined partially-greedy as those bases where 
$$\Vert f-P_A(f)\Vert \lesssim \Vert f-S_m(f)\Vert.$$
In the case when $\mathcal B$ is a Schauder basis, both definitions are equivalent. We work with \eqref{pg} inspired by the results obtained recently in \cite{BBL}.

In the following subsections, we give and analyze the main tools that we need for the characterization of partially-greediness in the context of quasi-Banach spaces.
\end{remark}

\subsection{The truncation operator}\label{trunc}
The following definitions follow from the truncation operator introduced in \cite{DKK2003}. Take $f\in\XX$, $A\subseteq \NN$ finite and $\varepsilon\equiv\lbrace\sgn(\ee_n^*(f))\rbrace$. Define, for $f\in\XX$ and $A$ a finite greedy set of $f$,
$$\mathcal U(f,A)=\min_{n\in A}\vert\ee_n^*(f)\vert\mathbf{1}_{\varepsilon A},$$
$$\mathcal T(f,A)=\mathcal U(f,A) + P_{A^c}(f).$$

These operators are called restricted truncation operator and truncation operator, respectively. Write the following quantities:

$$\Gamma_u=\Gamma_u[\BB,\XX]=\sup\lbrace \Vert \mathcal U(f,A)\Vert: \Vert f\Vert\leq 1,\, A\, \text{greedy set of}\, f\rbrace,$$
$$\Gamma_t=\Gamma_t[\BB,\XX]=\sup\lbrace \Vert \mathcal \mathcal T(f,A)\Vert: \Vert f\Vert\leq 1,\, A\, \text{greedy set of}\, f\rbrace.$$
%
%
%
%

\begin{theorem}[{\cite[Theorem 3.13, Proposition 3.14]{AABW}}]\label{operators}
	Let $\mathcal B$ a quasi-greedy basis in a quasi-Banach space $\XX$. Then,
	\begin{itemize}
		\item The restricted truncation operator $\mathcal U$ is uniformly bounded, that is, $\Gamma_u<\infty$. Also, if $\XX$ is a $p$-Banach space we have that $\Gamma_u \leq C_{qg}^2\eta_p(C_{qg})$, where, for $u>0$,
		$$\eta_p(u)=\min_{0<t<1}(1-t^p)^{-1/p}(1-(1+\mathbf{A}_p^{-1}u^{-1}t)^{-p})^{-1/p},$$
		and $\eta_p(C_{qg})\lesssim C_{qg}^{1+1/p}$.
		\item The truncation operator $\mathcal T$ is uniformly bounded, that is, $\Gamma_t<\infty$. Also, if $\XX$ is a $p$-Banach space, we have that $\Gamma_t \leq C_{qg}(1+C_{qg}^p\eta_p^p(C_{qg}))^{1/p}$.
	\end{itemize}
\end{theorem}

\begin{remark}
The estimate $\eta(C_{qg})\lesssim C_{qg}^{1+1/p}$ was given in \cite[Remark 3.9]{AABW}.
\end{remark}

\subsection{Properties about conservativeness}
Consider the following property introduced recently in \cite{BBL}.

\begin{definition}[\cite{BBL}]
We say that a basis $\mathcal B$ in a quasi-Banach space $\XX$ is \textbf{partially-symmetric for largest coefficients} if there exists a positive constant $C$ such that
\begin{eqnarray}\label{con}
\Vert f+\mathbf{1}_{\varepsilon A}\Vert \leq C\Vert f+\mathbf{1}_{\varepsilon' B}\Vert,
\end{eqnarray}
 for any pair of sets $A, B$, any $\varepsilon\in\Psi_A, \varepsilon'\in\Psi_B$ and any $f\in\XX$ such that $\vert A\vert\leq\vert B\vert$, $\max_{n\in \supp(f)}\vert\ee_n^*(f)\vert\leq 1$, $A<\supp(f)\cup B$ and $B\cap\supp(f)=\emptyset$. The smallest constant verifying \eqref{con} is denoted by $\Delta_{pl}=\Delta_{pl}[\BB,\XX]$ and we say that $\mathcal B$ is $\Delta_{pl}$-partially-symmetric for largest coefficients.
\end{definition}

\begin{definition}[\cite{DKKT2003}]
We say that a basis $\mathcal B$ in a quasi-Banach space $\XX$ is \textbf{super-conservative} if there exists a positive constant $C$ such that
\begin{eqnarray}\label{sup}
\Vert\mathbf{1}_{\varepsilon A}\Vert\leq C\Vert\mathbf{1}_{\varepsilon' B}\Vert,
\end{eqnarray}
for any pair of sets $A, B$ with $\vert A\vert\leq \vert B\vert$ and $A<B$, and any choice of signs $\varepsilon\in\Psi_A, \varepsilon'\in\Psi_B$. The smallest constant verifying \eqref{sup} is denoted by $\Delta_s=\Delta_s[\mathcal B, \mathbb X]$ and we say that $\mathcal B$ is $\Delta_s$-super-conservative.

If $\varepsilon\equiv\varepsilon'\equiv 1$ in \ref{sup}, we say that $\mathcal B$ is $\Delta$-\textbf{conservative}.
\end{definition}
Of course, $\Delta\leq \Delta_s$. In the following result, we characterize when $\mathcal B$ is partially-symmetric for largest coefficients.

\begin{proposition}\label{uni}
	Let $\mathcal B$ a basis in a quasi-Banach space $\XX$. 
	\begin{itemize}
		\item[a)] $\mathcal B$ is partially-symmetric for largest coefficients if and only if there exists a positive constant $\mathbf C$ such that
		\begin{eqnarray}\label{three}
		\Vert f\Vert \leq \mathbf C\Vert f-S_k(f)+\mathbf{1}_{\varepsilon B}\Vert,
		\end{eqnarray}
		for any finite set $B$, any sign $\varepsilon\in\Psi_B$, any element $f\in\mathbb X$ and any natural number $k$ such that $B\cap\supp(f)=\emptyset$, $k<\min B$ and $\max_{n\in \supp(f)}\vert\ee_n^*(f)\vert\leq 1$. Moreover, $\Delta_{pl}\leq \mathbf C\leq \mathbf{A}_p \Delta_{pl}$.
		\item[b)] $\mathcal B$ is partially-symmetric for largest coefficients if and only if $\mathcal B$ is super-conservative and quasi-greedy for largest coefficients. Moreover, if $\mathbb X$ is a $p$-Banach space,
		$$\Delta_s\leq \Delta_{pl},$$
		$$C_{ql}\leq (1+\Delta_{pl}^p)^{1/p},$$
		$$\Delta_{pl}\leq (1+(1+\Delta_s^p)C_{ql}^p)^{1/p}.$$
	\end{itemize}
\end{proposition}
\begin{proof}
	First of all, we prove a). Assume \eqref{three} and to show that $\mathcal B$ is partially-symmetric for largest coefficients, take $A, B, f, \varepsilon$ and $\varepsilon'$ as in \ref{con}. Define $f':=f+\mathbf{1}_{\varepsilon A}$ and take $k:=\max A$. Thus,
	\begin{eqnarray}
	S_k(f+\mathbf{1}_{\varepsilon A})=\mathbf{1}_{\varepsilon A}.
	\end{eqnarray}
	Hence, applying \eqref{three},
	\begin{eqnarray*}
	\Vert f+\mathbf{1}_{\varepsilon A}\Vert = \Vert f'\Vert &\leq& \mathbf C\Vert f'-S_k(f')+\mathbf{1}_{\varepsilon' B}\Vert\\
	&=&\mathbf C\Vert f+\mathbf{1}_{\varepsilon A}-S_k(f+\mathbf{1}_{\varepsilon A})+\mathbf{1}_{\varepsilon'B}\Vert\\
	&=&\mathbf C\Vert f+\mathbf{1}_{\varepsilon'B}\Vert.
	\end{eqnarray*}

\noindent Thus, $\BB$ is $\Delta_{pl}$-partially-symmetric for largest coefficients with $\Delta_{pl}\leq \mathbf C$. Assume now that $\mathcal B$ is $\Delta_{pl}$-partially-symmetric for largest coefficients. Take $B, f, k$ and $\varepsilon$ as in \eqref{three}, and define $A=\lbrace 1,...,k\rbrace$ with  $k<\min B$. Hence, using Proposition \ref{conv},
\begin{eqnarray*}
\Vert f\Vert=\Vert f-S_k(f)+S_k(f)\Vert &\leq& \mathbf{A}_p\sup\lbrace\Vert f-S_k(f)+\mathbf{1}_{\varepsilon A}\Vert: \varepsilon\in \Psi_A\rbrace\\
&\stackrel{\eqref{con}}{\leq}& \mathbf{A}_p \Delta_{pl}\Vert f-S_k(f)+\mathbf{1}_{\varepsilon' B}\Vert.
\end{eqnarray*}
The proof of a) is done.

Prove now b). Assume that $\mathcal B$ is $\Delta_{pl}$-partially-symmetric for largest coefficients. Take $f, A, \varepsilon\in\Psi_A$ as in \eqref{ql}. Then,
\begin{eqnarray*}
	\Vert \mathbf{1}_{\varepsilon A}\Vert^p&\leq& \Vert f+\mathbf{1}_{\varepsilon A}\Vert^p+\Vert f\Vert^p\\
	&\stackrel{\eqref{con}}{\leq}& \Vert f+\mathbf{1}_{\varepsilon A}\Vert^p+\Delta_{pl}^p\Vert f+\mathbf{1}_{\varepsilon A}\Vert^p\\
	&=&(1+\Delta_{pl}^p)\Vert f+\mathbf{1}_{\varepsilon A}\Vert^p.
\end{eqnarray*}
Thus, $\mathcal B$ is $C_{ql}$-quasi-greedy for largest coefficients with $C_{ql}\leq (1+\Delta_{pl}^p)^{1/p}$. The fact that $\mathcal B$ is super-conservative with $\Delta_s\leq \Delta_{pl}$ follows from the definition.

\noindent Assume now that $\mathcal B$ is $\Delta_s$-super-conservative and $C_{ql}$-quasi-greedy for largest coefficients. Take $f, A, B, \varepsilon, \varepsilon'$ as in \eqref{con}.
\begin{eqnarray*}
	\Vert f+\mathbf{1}_{\varepsilon A}\Vert^p&\leq& \Vert f+\mathbf{1}_{\varepsilon' B}\Vert^p+\Vert\mathbf{1}_{\varepsilon A}\Vert^p+\Vert\mathbf{1}_{\varepsilon' B}\Vert^p\\
	&\stackrel{\eqref{sup}}{\leq}& \Vert f+\mathbf{1}_{\varepsilon' B}\Vert^p+(1+\Delta_s^p)\Vert\mathbf{1}_{\varepsilon' B}\Vert^p\\
	&\stackrel{\eqref{ql}}{\leq}& \Vert f+\mathbf{1}_{\varepsilon' B}\Vert^p+(1+\Delta_s^p)C_{ql}^p\Vert f+\mathbf{1}_{\varepsilon' B}\Vert^p\\
	&=&(1+(1+\Delta_s^p)C_{ql}^p)\Vert f+\mathbf{1}_{\varepsilon' B}\Vert^p
\end{eqnarray*}
Thus, $\mathcal B$ is $\Delta_{pl}$-partially-symmetric for largest coefficients with $\Delta_{pl}\leq (1+(1+\Delta_s^p)C_{ql}^p)^{1/p}$ and the proof is over.
\end{proof}

\begin{question}
Is it possible to characterize super-conservativeness using conservativeness? That is, is there any property $\mathbf{X}$ such that if $\BB$ is conservative with the Property $\mathbf X$ then $\mathcal B$  is super-conservative? This question is still open in the case of Banach spaces.
\end{question}

\section{Characterization of partially-greediness}\label{four}
The following characterization will be useful to talk in the following section about renormings. Also, this theorem is given using a similar property that can be found in \cite{BB} and \cite{AABW}.
\begin{theorem}\label{th2}
Let $\mathcal B$ a basis in a quasi-Banach space $\mathbb X$. $\mathcal B$ is partially-greedy if and only if there exists a positive constant $\mathbf{D}$ such that
\begin{eqnarray}\label{one}
\Vert f\Vert \leq \mathbf{D} \Vert f-S_k(f)+z\Vert,
\end{eqnarray}
for every $f,z\in \XX$ and $k\in\NN$ such that $\vert\supp(z)\vert<\infty$, $k<\min \supp(z)$, $k\leq \vert\supp(z)\vert$, $\supp(f)\cap\supp(z)=\emptyset$ and $\max_{n\in \supp(f)}\vert\ee_n^*(f)\vert\leq \min_{n\in\supp(z)}\vert\ee_n^*(z)\vert.$
Moreover, $\mathbf{D}=C_{pg}$.
\end{theorem}

\begin{proof}
Assume \eqref{one}. Let $f\in\XX$, $A$ a greedy set of $f$ with cardinality $m\in\NN$ and take $k^*\in\NN$ such that
$$\inf_{k\leq m}\Vert f-S_k(f)\Vert=\Vert f-S_{k^*}(f)\Vert.$$

\noindent Define the elements $f':=f-P_A(f)$ and $z= P_A(f)-S_{k^*}(P_A(f))$. Of course, $z$ and $f'$ verifies the conditions of \eqref{one}. Hence,
\begin{eqnarray*}
\Vert f-P_A(f)\Vert = \Vert f'\Vert &\leq& \mathbf D\Vert f'-S_{k^*}(f)+z\Vert\\
&=&\mathbf D\Vert f-P_A(f)-S_{k^*}(f-P_A(f))+P_A(f)-S_{k^*}(P_A(f))\Vert\\
&=&\mathbf D\Vert f-P_A(f)-S_{k^*}(f)+S_{k^*}(P_A(f))+P_A(f)-S_{k^*}(P_A(f))\Vert\\
&=&\mathbf D\Vert f-S_{k^*}(f)\Vert.
\end{eqnarray*}
Thus, since this estimate works for any greedy set $A$ and any $k\leq\vert A\vert$, the basis is $C_{pg}$-partially-greedy with $C_{pg}\leq \mathbf D$.

\noindent Assume now that $\mathcal B$ is $C_{pg}$-partially-greedy basis and show \eqref{one}. Take $f,z,$ and $k$ as in \eqref{one}. Considering the element $g:=f+z$, $\supp(z)$ is a greedy set of $g$ with $m=\vert\supp(z)\vert$. Hence,
\begin{eqnarray*}
\Vert f\Vert=\Vert g-z\Vert&\leq& C_{pg}\inf_{k\leq m}\Vert g-S_k(f+z)\Vert\\
&\leq& C_{pg}\inf_{k: k<\min \supp(z), k\leq m}\Vert g-S_k(f+z)\Vert\\
&\leq& C_{pg}\Vert f-S_k(f)+z\Vert,
\end{eqnarray*} 
where the last inequality holds for any $k$ as in \eqref{one}. Hence, \eqref{one} is proved and the proof is over.
\end{proof}

\begin{theorem}\label{th1}
Let $\mathcal B$ a basis in a quasi-Banach space $\XX$. The following are equivalent:
\begin{itemize}
	\item[i)] $\mathcal B$ is partially-greedy.
	\item[ii)] $\mathcal B$ is quasi-greedy and partially-symmetric for largest coefficients.
	\item[iii)] $\mathcal B$ is quasi-greedy and super-conservative.
	\item[iv)] $\mathcal B$ is partially-symmetric for largest coefficients and the truncation operator is uniformly bounded.
\end{itemize}
Moreover, $\Delta_{pl}\leq C_{pg}$ and, in the particular case when $\XX$ is a $p$-Banach space, we also have
$$C_{pg}\leq \mathbf A_p\Delta_{pl}\Gamma_t,\;\; C_{qg}\leq 2^{1/p}C_{pg}.$$
\end{theorem}

\begin{proof}
i)$\Rightarrow$ ii). The argument to show that the basis is partially-symmetric for largest coefficients follows from \cite[Remark 1.10]{BBL}. To show that $\mathcal B$ is quasi-greedy, taking $k=0$ in the definition of partially-greediness, we obtain that
$$\Vert f-P_A(f)\Vert \leq C_{pg}\Vert f\Vert,$$
for any finite greedy set $A$ of $f$. Hence, $\mathcal B$ is quasi-greedy with $C_{qg}\leq 2^{1/p}C_{pg}$.

ii) $\Rightarrow$ iii), follows from the item b) of Proposition \ref{uni}.

iii)$\Rightarrow$ iv), follows from Theorem \ref{operators} and the item b) of Proposition \ref{uni}.

iv)$\Rightarrow$ i). Assume that $\mathcal B$ is $\Delta_{pl}$-partially-symmetric for largest coefficients and the truncation operator is uniformly bounded with constant $\Gamma_t$. Take $f\in\XX$, $P_{B}(f)$ with $B$ a greedy set of $f$ of cardinality $m$ and $A=\lbrace 1,...,k\rbrace$ with $k\leq m$. Then,
$$f-P_B(f)=P_{(A\cup B)^c}(f-S_k(f))+P_{A\setminus B}(f).$$
Thus, taking $t:=\min_{n\in B}\vert\ee_n^*(f)\vert$, applying Proposition \ref{conv},

\begin{eqnarray}\label{u}
\Vert f-P_B(f)\Vert &\leq& \mathbf{A}_p\sup_{\eta\in \Psi_{A\setminus B}}\Vert P_{(A\cup B)^c}(f-S_k(f))+t\mathbf{1}_{\eta(A\setminus B)}\Vert
\end{eqnarray}

\noindent If we select $\varepsilon\equiv\lbrace\sgn(\ee_n^*(f))\rbrace$, applying the fact that $\BB$ is $\Delta_{pl}$-partially-symmetry for largest coefficients in combination with \eqref{u},
\begin{eqnarray*}
\Vert f-P_B(f)\Vert &\stackrel{\eqref{con}+\eqref{u}}{\leq}& \mathbf{A}_p\Delta_{pl}\Vert P_{(A\cup B)^c}(f-S_k(f))+t\mathbf{1}_{\varepsilon(B\setminus A)}\Vert\\
&=&\mathbf{A}_p\Delta_{pl}\Vert\mathcal T(f-S_k(f), B\setminus A)\Vert\\
&\stackrel{\text{Theorem}\, \ref{operators}}{\leq}& \mathbf{A}_p\Delta_{pl}\Gamma_t\Vert f-S_k(f)\Vert.
\end{eqnarray*}
Thus, since the last estimate works for any $k\leq \vert B\vert$ and any finite greedy set $B$, $\mathcal B$ is $C_{pg}$-partially-greedy with $C_{pg}\leq \mathbf{A}_p\Delta_{pl}\Gamma_t$.
\end{proof}

The above theorem gives estimates for the partially-greedy constant
using the partially-symmetry for largest coefficients constant and the truncation operator constant. We complete that result providing estimates of $C_{pg}$ in terms of the quasi-greedy constant
of the basis and some constants related to conservative-like properties.

\begin{theorem}
Let $\mathcal B$ be a basis of a $p$-Banach space $\mathbb X$. Assume that $\mathcal B$ is $C_{qg}$-quasi-greedy. Then,
\begin{itemize}
	\item[i)] If $\mathcal B$ is $\Delta$-conservative, then $\mathcal B$ is $C_{pg}$-partially-greedy with
	$$C_{pg}\leq C_{qg}(1+(\mathbf{A}_p\mathbf{B}_p\Delta C_{qg}\eta_p(C_{qg}))^p)^{1/p}.$$
	\item[ii)] If $\mathcal B$ is $\Delta_s$-super-conservative, then $\mathcal B$ is $C_{pg}$-partially-greedy with
	$$C_{pg}\leq C_{qg}(1+(\mathbf{A}_p\Delta_s\eta_p(C_{qg}))^p)^{1/p}.$$
	\item[iii)] If $\mathcal B$ is $\Delta_{pl}$-partially-symmetric for largest coefficients, then $\mathcal B$ is $C_{pg}$-partially-greedy with
	$$C_{pg}\leq \mathbf{A}_p\Delta_{pl}C_{qg}(1+C_{pg}^p\eta_p^p(C_{qg}))^{1/p}.$$
\end{itemize}
\end{theorem}

\begin{proof}
As in Theorem \ref{th1}, take $f\in\XX$, $P_{B}(f)$ with $B$ a finite greedy set of cardinality $m$ and $A=\lbrace 1,...,k\rbrace$ with $k\leq m$. Then,
$$f-P_B(f)=P_{(A\cup B)^c}(f-S_k(f))+P_{A\setminus B}(f).$$

i) On the one hand,  we have that
\begin{eqnarray}\label{p1}
\Vert P_{(A\cup B)^c}(f-S_k(f))\Vert \leq  2^{1/p}C_{qg}\Vert f-S_k(f)\Vert.
\end{eqnarray}
Now, we only need to control $\Vert P_{A\setminus B}(f)\Vert$. 
\begin{eqnarray}\label{p2}
\nonumber\Vert P_{A\setminus B}(f)\Vert &\stackrel{\text{Proposition}\,  \ref{conv}}{\leq}& \mathbf{B}_p\Delta\max_{n\in A\setminus B}\vert\ee_n^*(f)\vert\Vert\mathbf{1}_{B\setminus A}\Vert\\\nonumber
&\leq&\mathbf{B}_p\Delta\min_{n\in B\setminus A}\vert\ee_n^*(f)\vert\Vert\mathbf{1}_{B\setminus A}\Vert\\
&\leq& \mathbf{A}_p\mathbf{B}_p\Delta C_{qg}^2\eta_p(C_{qg})\Vert f-S_k(f)\Vert,
\end{eqnarray}
where the last inequality is due to \cite[Theorem 3.10]{AABW}. Thus, by \eqref{p1} and \eqref{p2}, we obtain the result.

ii) Taking $\varepsilon=\lbrace\sgn(\ee_n^*(f))\rbrace$,

\begin{eqnarray}\label{p3}
\nonumber\Vert P_{A\setminus B}(f)\Vert &\stackrel{\text{Proposition}\,  \ref{conv}}{\leq}& \mathbf{A}_p\Delta_s\max_{n\in A\setminus B}\vert\ee_n^*(f)\vert\Vert\mathbf{1}_{\varepsilon(B\setminus A)}\Vert\\\nonumber
&\leq&\mathbf{A}_p\Delta_s\min_{n\in B\setminus A}\vert\ee_n^*(f)\vert\Vert\mathbf{1}_{\varepsilon(B\setminus A)}\Vert\\
&\stackrel{\text{Theorem}\, \ref{operators}}{\leq}& \mathbf{A}_p\Delta_sC_{qg}\eta_p(C_{qg})\Vert f-S_k(f)\Vert.
\end{eqnarray}
By \eqref{p1} and \eqref{p3}, we obtain the estimate.

iii) Follows from iv) of Theorem \ref{th1} and Theorem \ref{operators}.

The proof is done.
\end{proof}

\section{Renormings of partially-greedy bases}\label{five}

In \cite{DKOSZ2014}, the authors continue with one of the most difficult problems in greedy approximation theory: renorming Banach spaces with greedy bases. In this paper, the authors characterized $1$-greedy bases (generalizing the result proved in \cite{AW2006}) and, also, proved that, for a fixed $\varepsilon>0$, it is possible to find a renorming in $L_p$, $1<p<\infty$, such that the Haar system is $(1+\varepsilon)$-greedy, but it is unsolved if it is possible or not to get the constant $1$. The topic of renorming Banach spaces with greedy bases continue nowadays and one recent paper talking about this theory is \cite{AAW2018}. One of the most important keys in all of these papers is the use of convexity! Also, in Banach spaces, it is well known that a renorming $\Vert \cdot\Vert_0$ of $(\XX,\Vert\cdot\Vert)$ has the form
$$\Vert f\Vert_0=\max\lbrace a\Vert f\Vert, \Vert T(f)\Vert_{\mathbb Y}\rbrace,$$
for some $0<a<\infty$ and some bounded linear operator $T$ from $\XX$ into a Banach space $\mathbb Y$. 

Renorming quasi-Banach spaces with greedy bases has been recently studied in \cite{AABW}, where one of the most important tools is the following lemma:

\begin{lemma}[{\cite[Lemma 11.1]{AABW}}]\label{qu}
Let $(\XX,\Vert \cdot\Vert)$ be a quasi-Banach space $\mathbb X$. Assume that $\Vert\cdot\Vert_0:\XX\longrightarrow [0,+\infty)$ is such that
\begin{enumerate}
	\item $\Vert t\, f\Vert_0=\vert t\vert\,\Vert f\Vert_0$, for every $t\in\FF$, for every $f\in\XX$.
	\item $\Vert f\Vert_0\approx \Vert f\Vert$.
\end{enumerate}
Then, $\Vert \cdot\Vert_0$ is a renorming of $\Vert \cdot\Vert$.
\end{lemma}

This lemma allows us to have renormings of quasi-Banach spaces based on non-linear operators! Here, we follow the ideas of \cite[Theorem 11.3]{AABW} to give a renorming such that $C_{pg}=1$.

\begin{theorem}
Let $\mathcal B$ a partially-greedy basis of a quasi-Banach space $\mathbb X$. Then, there is a renorming of $\XX$ with respect to which $C_{pg}=1$.
\end{theorem}

\begin{proof}
Based on Theorem \ref{th2}, we introduce the following quantity:
$$\Vert f\Vert_a = \inf\lbrace \Vert f-S_k(f)+z\Vert : (k,z)\in\mathcal D(f)\rbrace,$$
where $(k,z)\in\mathcal D(f)$ if  $\vert\supp(z)\vert<\infty$, $k<\min \supp(z)$, $\supp(z)\cap\supp(f)=\emptyset$, $k\leq \vert \supp(z)\vert$ and $\max_{n\in\supp(f)}\vert\ee_n^*(f)\vert\leq\min_{n\in\supp(z)}\vert\ee_n^*(z)\vert$.
By Theorem \ref{th2}, $\Vert f\Vert_a \approx \Vert f\Vert$ for $f\in\XX$, so applying Lemma \ref{qu}, $\Vert \cdot\Vert_a$ is a renorming of $\Vert \cdot\Vert$. 

\noindent Take now $(k,z)\in \mathcal D(f)$ and write $g:=f-S_k(f)+z$ and $A=\lbrace 1,...,k\rbrace$. Take now $(m,y)\in\mathcal D(g)$ and define $B=\lbrace 1,...,m\rbrace$, $B_1= B\cap (\supp(f-S_k(f))$ and $B_2=B\cap \supp(z)$. Then,
$$g-S_m(g)=f-P_{A\cup B_1}(f)+z-P_{B_2}(z).$$
It is clear that $\supp(z-P_{B_2}(z))\cap \supp(y)=\emptyset$ and
\begin{eqnarray*}
p=\vert A\cup B_1\vert &=& k+m-\vert B_2\vert\\
&\leq& \vert\supp(z)\vert+\vert\supp(y)\vert-\vert B_2\vert\\
&=& \vert\supp(z-P_{B_2}(z))\vert+\vert\supp(y)\vert\\
&=&\vert\supp(z-P_{B_2}(z)+y)\vert.
\end{eqnarray*}
We infer that $(p,z-P_{B_2}(z)+y)\in\mathcal D(f)$ and 
$$\Vert f\Vert_a\leq \Vert f-S_p(f)+z-P_{B_2}(z)+y\Vert = \Vert g-S_m(g)+y\Vert.$$
Taking the infimum over $(m,y)$ we get $\Vert f\Vert_a\leq \Vert g\Vert_a$ and, based on the estimates of Theorem \ref{th2}, the proof is done.
\end{proof}
\newpage
\section*{Annex: Summary of  the most important constants}
\begin{table}[ht]
	\begin{center}
		\begin{tabular}{c c c}\hline
			& & \\
			{\bf Symbol}& {\bf Name of constant} &  {\bf Ref. equation}\\
			& & \\
			\hline
			& & \\
			$\Delta_{pl}$ & Partially-symmetry for largest coeffs. constant & \eqref{con} \\ 
			& & \\
			$\Delta_s$ & Super-conservativeness constant& \eqref{sup} \\ 
			& & \\
			$\Delta $ & Conservativeness constant & \eqref{sup} \\ 
			& & \\
			$C_{qg}$         & Quasi-greedy constant & \eqref{qg} \\
			& & \\
			$C_{ql}$ & Quasi-greedy for largest coeffs. constant & \eqref{ql} \\ 
			& & \\
			$C_{pg}$ & Partially-greedy constant &  \eqref{pg} \\ 
			& & \\
			$\Gamma_u$ & Restricted truncation operator constant &  Section \ref{trunc} \\ 
			& & \\
			$\Gamma_t$ & Truncation operator constant &  Section \ref{trunc} \\ 
		\end{tabular}
	\end{center}
\end{table}

\end{document}